\documentclass[12pt,reqno]{amsart}

\usepackage{graphicx}
\graphicspath{ {./images/} }
\usepackage{bm}

\usepackage{color}

\usepackage{amssymb}

\usepackage{amsfonts}

\usepackage{amsmath}

\usepackage[all]{xy}

\newcommand{\PP}{\mathbb{P}}
\newtheorem{theorem}{Theorem}[section]

\newtheorem{corollary}[theorem]{Corollary}

\newtheorem{lemma}[theorem]{Lemma}

\newtheorem{proposition}[theorem]{Proposition}

\newtheorem{Definition}[theorem]{Definition}

\newtheorem{Example}[theorem]{Example}

\newtheorem{Remark}[theorem]{Remark}

\newcommand{\arccosh}{\operatorname{arcosh}}
\newenvironment{remark}{\begin{Remark}\begin{em}}{\end{em}\end{Remark}}

\DeclareMathOperator{\tr}{tr}

\address{Sejong Kim \\ Department of Mathematics, Chungbuk National University, Cheongju 28644, Korea}
\email{skim@chungbuk.ac.kr}

\address{Vatsalkumar N. Mer \\ Institute for Industrial and Applied Mathematics, Chungbuk National University, Cheongju 28644, Korea}
\email{vnm232657@gmail.com}

\begin{document}

\title[A multivariable mean equation arising from the spectral geometric mean]{A multivariable mean equation arising from the spectral geometric mean}

\author{Sejong Kim and Vatsalkumar N. Mer}

\date{}

\begin{abstract}
In the 1980s, Kubo and Ando introduced operator means on $\mathbb{P}$, the open convex cone of positive definite operators. One significant example is the weighted geometric mean
\begin{displaymath}
A \sharp_{t} B = A^{1/2} (A^{-1/2} B A^{-1/2})^{t} A^{1/2}, \qquad A,B \in \mathbb{P}.
\end{displaymath}
The Karcher mean serves as a natural multivariable extension of this mean by minimizing the sum of squared Riemannian trace distances of positive definite matrices. It coincides a unique positive definite solution to the Karcher equation, which allows us to define the Karcher mean on $\mathbb{P}$.
The weighted spectral geometric mean is defined as another geometric mean of two positive definite operators as follows:
\begin{displaymath}
A \natural_t B = (A^{-1} \sharp B)^{t} A (A^{-1} \sharp B)^{t},
\end{displaymath}
where $A \sharp B = A \sharp_{1/2} B$.
In this paper, we make an initial attempt to formulate a multivariable
spectral geometric mean through a nonlinear equation.
In the two-variable case, the unique positive definite solution of this equation is precisely
the spectral geometric mean. However, in the multi-variable case, the
equation need not have a unique solution. We study properties of its
solutions and compare them with other least squares means of positive definite matrices.
Recently, a new theory of alternative means for positive definite operators has been developed, which includes the spectral geometric mean and the Wasserstein mean. We also consider
multivariable equation arising from the alternative means.

\vspace{5mm}

\noindent {\bf Mathematics Subject Classification} (2020): primary  47A63, 47A64, 15B48 , secondary 53C20.

\noindent {\bf Keywords}: Positive definite operator, spectral geometric mean, Wasserstein mean, alternative mean
\end{abstract}

\maketitle

\section{Introduction}
% Averaging multiple data sources is one of the most basic and essential methods in data science.
%Positive definite matrices play a crucial role in machine learning, computer vision, and image processing, offering a strong foundation for data representation.

Let $B(\mathcal{H})$ be the Banach space of all bounded linear operators on a Hilbert space $\mathcal{H}$. Let $S(\mathcal{H}) \subset B(\mathcal{H})$ be the closed subspace of all self-adjoint linear operators. We denote as $\mathbb{P} \subset S(\mathcal{H})$ the open convex cone of all invertible positive definite operators. For $X, Y \in S(\mathcal{H})$, we define as $X \leq Y$ if $Y - X$ is positive semi-definite, and $X < Y$ if $Y - X$ is positive definite. It is a partial order, known as the L{\"o}wner order. In the finite-dimensional setting $\mathcal{H} = \mathbb{C}^{n}$, we denote as $\mathbb{H}_n$ and $\mathbb{P}_n$ the real vector space of all $n \times n$ Hermitian matrices and the open convex cone of all $n \times n$ positive definite matrices respectively.

The concept of the geometric mean of two positive definite matrices is first appeared in a mathematical physics context in the work of Pusz and Woronowicz \cite{PW}.    For \( A, B \in \mathbb{P} \), their geometric mean is defined as
\[
A \sharp B = A^{1/2} (A^{-1/2} B A^{-1/2})^{1/2} A^{1/2},
\]
which is now commonly referred to as the metric geometric mean. Later, Kubo and Ando \cite{ka} developed a systematic and axiomatic theory of operator means for positive definite operators.
A binary operation $\sigma$  on positive definite operators is called \emph{operator mean}
if it satisfies the following conditions:
\begin{enumerate}
\item $A \sigma A =A$ for every $A \in \mathbb{P}$.
\item If $A \leq C$ and $B \leq D$, then $A \sigma B\leq C \sigma D$.
\item $T(A \sigma B) T^* = TAT^*\sigma TBT^*$ for any invertible operator $T$.
\item $\sigma$ is continuous.
\end{enumerate}
One of their fundamental results is to characterize all such operator means in terms of operator monotone functions:
\begin{theorem}\label{Opm}
For every operator mean $\sigma$ on $\mathbb{P}$, there exists an operator monotone function $f: (0,\infty) \rightarrow (0,\infty)$, normalized by $f(1)=1$, such that $A \sigma B = A^{1/2} f(A^{-1/2} B A^{-1/2})A^{1/2}$ for every $A ,B \in \mathbb{P}$.
\end{theorem}
We customarily denote this mean by
$
A \sigma_f B := A^{1/2} f\!\left(A^{-1/2} B A^{-1/2}\right) A^{1/2}.
$
An important example in the Kubo–Ando operator means  is the weighted geometric mean of $A$ and $B$ in $\mathbb{P}$ given by
\begin{displaymath}
A \sharp_{t} B = A^{1/2} (A^{-1/2} B A^{-1/2})^{t} A^{1/2}, \ t \in [0,1].
\end{displaymath}
Note that $A \sharp_{1/2} B = A \sharp B$ is a unique solution $X \in \mathbb{P}$ of the Riccati equation $X A^{-1} X = B$. Various approaches to extend the geometric mean from two to multi-variables of positive definite matrices have been developed.

% Ando, Li, and Mathias \cite{ALM} proposed a symmetrization procedure using an inductive approach to define a multivariable generalization of the geometric mean, now known as the ALM geometric mean. Bini, Meini, and Poloni \cite{BMP} introduced another generalization called the BMP geometric mean, which also preserves desirable matrix properties for the geometric mean.
% \subsection{Least squares means}
Moakher \cite{MO} and  Bhatia and Holbrook \cite{BH} suggested extending the geometric mean of positive definite matrices to $m$-points by taking the unique minimizer of the sum of squares of the Riemannian trace distance, called a \emph{least squares mean}. The Riemannian trace distance on $\mathbb{P}_n$ is defined as
\begin{equation}  \label{R_metric}
d_R (A,B): = \parallel \log A^{-1/2} B  A^{-1/2} \parallel_2
\end{equation}
The Riemannian manifold $(\mathbb{P}_n, d_R)$ is an NPC-space (i.e., a complete metric space satisfying the semi-parallelogram
law). The unique geodesic connecting $A$ and $B$ for the Riemannian trace distance $d_R$ is given by $\gamma(t) = A \sharp_{t} B, t \in \mathbb{R}$.
The least squares mean or Karcher mean (barycentre) of $A_1, A_2, \ldots ,A_m \in \mathbb{P}_n$ is defined as
\begin{equation} \label{Karcher-mean}
\Lambda (\omega; A_1, A_2, \ldots ,A_m)  := \underset{X \in \mathbb{P}_n }{\arg\min} \sum\limits_{i=1}^{m} w_i d_R^2(X,A_i),
\end{equation}
where $\omega = (w_1, \ldots, w_m)$ is a positive probability vector: $w_i > 0$ and $\sum\limits_{i=1}^{m} w_i = 1$.
The Karcher mean $\Lambda (\omega ; A_1, A_2, \ldots ,A_m) $ coincides with the unique solution $X \in \mathbb{P}_n$  of the Karcher equation
\begin{equation} \label{Karcher-eq}
\sum_{i=1}^{m} w_{i} \text{log} (X^{-1/2} A_{i} X^{-1/2}) = 0.
\end{equation}

\begin{theorem} \label{T:properties}
Let $\mathbb{A} = (A_{1}, \ldots, A_{m}), \mathbb{B} = (B_{1}, \ldots, B_{m}) \in \mathbb{P}_{n}^{m}$, and let $\omega = (w_{1}, \dots, w_{m})$ be a positive probability vector. Then
\begin{itemize}
\item[(P1)] $\mathrm{ ( Consistency \ with \ scalars)} $ $\Lambda (\omega ; \mathbb{A}) =  A_1^{w_1} A_2^{w_2} \cdots A_m^{w_m}$ if $A_i$'s commute;

\item[(P2)] $\mathrm{( Joint \ homogeneity)}$ $\Lambda (\omega; a_1 A_1, a_2 A_2, \ldots , a_m A_m) = a_1^{w_1} a_2^{w_2} \cdots a_m^{w_m} \Lambda ( \omega; \mathbb{A})$ for any $a_i > 0$;

\item[(P3)] $\mathrm{(Permutation \ invariance)}$ $\Lambda (\omega_{\sigma}; \mathbb{A}_{\sigma}) = \Lambda (\omega; \mathbb{A})$ for any permutation $\sigma$ on $\{ 1, \dots, m \}$, where
\begin{displaymath}
\omega_{\sigma} := (w_{\sigma(1)}, w_{\sigma(2)}, \dots, w_{\sigma(m)}), \
\mathbb{A}_{\sigma} := (A_{\sigma(1)}, A_{\sigma(2)}, \dots, A_{\sigma(m)});
\end{displaymath}

\item[(P4)] $\mathrm{(Monotonicity)}$ $\Lambda (\omega; \mathbb{A}) \leq \Lambda (\omega; \mathbb{B})$ whenever $A_{i} \leq B_{i}$ for all $i = 1, \dots, m$;

\item[(P5)] $\mathrm{(Continuity)}$ $\Lambda (\omega; \cdot): \mathbb{P}_{m}^{n} \to \mathbb{P}_{m}$ is contractive for the Thompson metric:
\begin{displaymath}
d_{T}(\Lambda (\omega; \mathbb{A}), \Lambda (\omega; \mathbb{B})) \leq \sum_{i=1}^{m} w_{i} d_{T}(A_{i}, B_{i}),
\end{displaymath}
where $d_T (A,B): = \Vert \log A^{-1/2} B  A^{-1/2} \Vert$ for an operator norm $\Vert \cdot \Vert$;

\item[(P6)] $\mathrm{(Congruence \ invariance)}$ $\Lambda (\omega; S \mathbb{A} S^{*}) = S \Lambda (\omega; \mathbb{A}) S^{*}$ for any $S \in \mathrm{GL}_{n}$, where
\begin{displaymath}
S \mathbb{A} S^{*} := (S A_{1} S^{*}, S A_{2} S^{*}, \dots, S A_{m} S^{*});
\end{displaymath}

\item[(P7)] $\mathrm{(Joint \ concavity)}$ $\Lambda (\lambda \mathbb{A} + (1 - \lambda) \mathbb{B}) \geq \lambda \Lambda (\mathbb{A}) + (1 - \lambda) \Lambda(\mathbb{B})$ for any $\lambda \in [0,1]$;

\item[(P8)] $\mathrm{(Self-duality)}$ $\Lambda (\omega; \mathbb{A}^{-1}) = \Lambda(\omega; \mathbb{A})^{-1}$;

\item[(P9)] $\mathrm{(Determinantal identity)}$ $\displaystyle \det \Lambda (\omega; \mathbb{A}) = \prod_{i=1}^{m} (\det A_{i})^{w_{i}}$;

\item[(P10)] $\mathrm{(AGH \ weighted \ mean \ inequalities)}$
\begin{displaymath}
\mathcal{H}(\omega; \mathbb{A}) := \left[ \sum_{i=1}^{m} w_{i} A_{i}^{-1} \right]^{-1} \leq \Lambda(\omega; \mathbb{A}) \leq \sum_{i=1}^{m} w_{i} A_{i} =: \mathcal{A}(\omega; \mathbb{A}).
\end{displaymath}
\end{itemize}
\end{theorem}

%Before stating the main result, we introduce some notation:
%
%\begin{itemize}
%    \item Let \( \mathcal{P}^{0}(\mathbb{P}) \) denote the set of Borel probability measures on \( \mathbb{P} \) with bounded support.
%    \item Define the class
%    \[
%    \mathcal{L} := \left\{ f : (0, \infty) \rightarrow \mathbb{R} \ \middle| \ f \text{ is operator monotone}, \ f(1) = 0, \ f'(1) = 1 \right\}.
%    \]
%\end{itemize}
%
%We now state the key result:
%
%\begin{theorem}[Pálfia, 2016]
%Let \( f \in \mathcal{L} \) and \( \sigma \in \mathcal{P}^{0}(\mathbb{P}) \). Then, the generalized Karcher equation
%\[
%\int_{\mathbb{P}} f(X^{-1/2} A X^{-1/2}) \, d\sigma(A) = 0
%\]
%has a unique solution \( \Lambda_f(\sigma) = X \in \mathbb{P} \). T
%\end{theorem}
%
%\begin{corollary}[Pálfia, 2016]
%Let \( \nu \in \mathcal{P}^{0}([0,1]) \), and let \( \sigma \in \mathcal{P}^{0}(\mathbb{P}) \) be supported on the two points \( \{A\} \) and \( \{B\} \), where \( A, B \in \mathbb{P} \), with weights \( (1 - t) \) and \( t \in [0,1] \), respectively. Then \( \Lambda(\nu \times \sigma) \) corresponds to a Kubo--Ando mean of \( A \) and \( B \).
%\end{corollary}

Another important least squares mean is the Wasserstein mean $\Omega$ of $\mathbb{A} = (A_{1}, \ldots, A_{m}) \in \mathbb{P}_{n}^{m}$ for the Bures-Wasserstein metric:
\begin{displaymath}
\Omega(\omega; \mathbb{A}) = \underset{X \in \mathbb{P}_{n}}{\arg \min} \, \sum_{i=1}^{m} w_{i} d_{W}^{2}(X, A_{i}),
\end{displaymath}
where $d_{W}(A, B) = \left[ \tr (A + B - 2(A^{1/2} B A^{1/2})^{1/2}) \right]^{1/2}$ denotes the Bures-Wasserstein distance between $A$ and $B$.
Note from \cite{ABCM, BJL} that the Wasserstein mean coincides with the unique positive definite solution $X$ of the following matrix equation
 \begin{equation} \label{Wass-eq}
\sum_{i=1}^{m} w_{i} A_{i} \sharp X^{-1} = I,
\end{equation}
and equivalently
\begin{displaymath}
X = \sum_{i=1}^{n} w_{i} (X^{1/2} A_{i} X^{1/2})^{1/2}.
\end{displaymath}
The Wasserstein mean of $A$ and $B$ is given by
\begin{align}
A \diamond_{t} B & = (1-t)^{2} A + t^{2} B + t(1-t) \left[ A (A^{-1} \sharp B) + (A^{-1} \sharp B) A \right] \\
&= (I \nabla_t (A^{-1} \sharp B)) \, A \, (I \nabla_t (A^{-1} \sharp B)), \label{wass}
\end{align}
where $X \nabla_t Y := (1 - t)X + tY$ is the weighted arithmetic mean. The second expression \eqref{wass} for the Wasserstein mean appears in \cite{HK}. In the following we denote as $\Delta_{m}$ the set of all positive probability vectors in $\mathbb{R}^{m}$.

%Note that $A \diamond_{t} B = (1-t)^{2} A + t^{2} B + t(1-t) [A (A^{-1} \sharp B) + (A^{-1} \sharp B) A]$ for $t \in [0,1]$ is the two-variable Wasserstein mean
%\begin{displaymath}
%A \diamond_{t} B = \underset{X \in \mathbb{P}_{m}}{\arg \min} \, (1-t) d_{W}^{2}(X, A) + t d_{W}^{2}(X, B).
%\end{displaymath}

\begin{lemma} \cite{ABCM} \label{L:Wass-iteration}
Let $\omega = (w_{1}, \dots, w_{m}) \in \Delta_{m}$ and $\mathbb{A} = (A_{1}, \dots, A_{m}) \in \mathbb{P}_{n}^{m}$. For every $X_{0} \in \mathbb{P}_{n}$ the sequence $X_{k+1} = K(X_{k})$, constructed inductively from the map $K: \mathbb{P}_{n} \to \mathbb{P}_{n}$ defined by
\begin{displaymath}
K(A) := A^{-1/2} \left[ \sum_{i=1}^{m} w_{i} (A^{1/2} A_{i} A^{1/2})^{1/2} \right]^{2} A^{-1/2}, \quad A \in \mathbb{P}_{n},
\end{displaymath}
converges to $\Omega(\omega; \mathbb{A})$. Furthermore, for all natural number $k$
\begin{displaymath}
\tr X_{k} \leq \tr X_{k+1} \leq \tr \Omega(\omega; \mathbb{A}).
\end{displaymath}
\end{lemma}
Note from \cite{HK17, HK19} that Karcher and Wasserstein means satisfy the extended Lie-Trotter formula:
\begin{displaymath}
\lim_{s \to 0} \Lambda(\omega; \mathbb{A}^{s})^{1/s} = \exp \left( \sum_{i=1}^{m} w_{i} \log A_{i} \right) = \lim_{s \to 0} \Omega(\omega; \mathbb{A}^{s})^{1/s},
\end{displaymath}
where the middle expression is known as the log-Euclidean mean $\textrm{LE}(\omega; \mathbb{A})$.
Furthermore, the following relationships among Karcher, log-Euclidean and Wasserstein means have been shown in \cite{BJL, BJL-1}:
\begin{equation} \label{E:multi-log-m}
\mathcal{H}(\omega; \mathbb{A}) \leq \Lambda(\omega; \mathbb{A}) \prec_{\log} \textrm{LE}(\omega; \mathbb{A}) \prec_{w \log} \Omega(\omega; \mathbb{A}) \leq \mathcal{A}(\omega; \mathbb{A}).
\end{equation}
%where
%\begin{center}
%$\displaystyle \mathcal{A}(\omega; \mathbb{A}) = \sum_{i=1}^{n} w_{i} A_{i}$ \quad and \quad $\displaystyle \mathcal{H}(\omega; \mathbb{A}) = \left( \sum_{i=1}^{n} w_{i} A_{i}^{-1} \right)^{-1}$
%\end{center}
%are known as the weighted arithmetic and harmonic means respectively.

The paper is organized as follows. In Section 2, we recall the two-variable
spectral geometric mean and related recent results. We then introduce a
multivariable equation arising from the spectral geometric mean
\eqref{E:n-equation} and show that this equation need not have a unique
solution in the multi-variable case. In Section 3, we study properties of
the solutions of \eqref{E:n-equation} and compare them with
the power mean and the log-Euclidean mean. Section 4 is devoted to
generalized Karcher means and two-variable alternative means. We also consider
multivariable equations arising from these alternative means and conclude with an open
problem.

\section{The multivariable mean equation}

The weighted spectral geometric mean of $A, B \in \mathbb{P}_{n}$ is defined in \cite{LL} as
\begin{align}\label{spe}
A \natural_t B &= (A^{-1} \sharp B)^t A (A^{-1} \sharp B)^t, \quad t \in [0,1] \\
               &= (I \sharp_t (A^{-1} \sharp B)) A (I \sharp_t (A^{-1} \sharp B)).
\end{align}
Note from \cite{FP} that $A \natural B = A \natural_{1/2} B$ has the same spectrum with $(A B)^{1/2}$. It is not a Kubo-Ando's operator mean, which means that the monotonicity with respect to the L{\"o}wner   order does not hold. It is interesting to note that the second expression of Wasserstein mean in \eqref{wass} shows a structural similarity to the
spectral geometric mean, with the weighted arithmetic mean replacing the metric
geometric mean. Recently, the weighted spectral geometric mean has been further studied from
several viewpoints. In particular, Gan, Kim and Mer investigated
characterizations and the linearity problem of the weighted spectral geometric
mean \cite{GKM}. It has been shown in \cite[Theorem 3.2]{GK} that the weighted spectral geometric mean $A \natural_{t} B$ is a unique positive definite solution $X$ to the equation
\begin{equation} \label{E:2-equation}
(X^{-1} \sharp A) \sharp_{t} (X^{-1} \sharp B) = I,
\end{equation}
which is equivalent to
\begin{equation}
t \log(A \sharp X^{-1}) + (1 - t) \log(B \sharp X^{-1}) = 0.
\end{equation}
The following log-majorization relationships have been shown in \cite{AH, GH24, GT24}:
\begin{equation} \label{E:log-majorization}
A \sharp_{t} B \prec_{\log} e^{(1-t) \log A + t \log B} \prec_{\log} A \natural_{t} B \preceq A \diamond_{t} B,
\end{equation}
where $\preceq$ denotes the near order: $A \preceq B$ if and only if $A^{-1} \sharp B \geq I$. Furthermore, the following geodesic property has been verified:
\begin{theorem}[\cite{GK}]
The weighted spectral geometric mean \(A \natural_t B\) is a geodesic with respect to the semi-metric  $d_{S}(A, B) = \| \log(A^{-1} \sharp B) \|.$
That is, for all \(s, t \in \mathbb{R}\),
\[
d_{S}(A \natural_s B, A \natural_t B) = |s - t| \, d_{S}(A, B).
\]
\end{theorem}

We study in this section the existence and uniqueness of a positive definite solution to the matrix equation
\begin{equation} \label{E:n-equation}
 \Psi (X) = \Lambda \big(\omega; A_{1} \sharp X^{-1}, A_{2} \sharp X^{-1}, \ldots, A_{m} \sharp X^{-1}\big) - I = 0.
\end{equation}
Note that \eqref{E:n-equation} is a multi-variable extension of \eqref{E:2-equation}, and is equivalent to
\[
\Lambda\!\left(\omega; (X^{1/2} A_{1} X^{1/2})^{1/2}, \ldots, (X^{1/2} A_{m} X^{1/2})^{1/2}\right) = X,
\]
by the congruence invariance of the Karcher mean (P6).
By the Karcher equation, \eqref{E:n-equation} can also be written equivalently as
\begin{align}\label{phi}
\Phi(X) = \sum_{i=1}^{m} w_i \log(A_i \sharp X^{-1}) = 0.
\end{align}

% \begin{remark} \label{R:Karcher-Sp}
% By the Karcher equation, \eqref{E:n-equation} can also be written equivalently as
% \begin{align}\label{phi}
% \Phi(X) = \sum_{i=1}^{m} w_i \log(A_i \sharp X^{-1}) = 0.
% \end{align}
% We denote as $\mathcal{S}(\omega; \mathbb{A})$ the unique solution $X \in \mathbb{P}_{m}$ to \eqref{phi} and call it the multi-variable spectral geometric mean of $\mathbb{A} = (A_{1}, \dots, A_{n}) \in \mathbb{P}_{m}^{n}$.
% Moreover, $\mathcal{S}(\omega; \mathbb{A})$ is the unique solution $X \in \mathbb{P}_{m}$ to the equation
% \begin{displaymath}
% \Lambda(\omega; (X^{-1} \sharp A_{1})^{p}, \dots, (X^{-1} \sharp A_{n})^{p}) = I
% \end{displaymath}
% for all $p > 0$.
% \end{remark}

\begin{remark} \label{brouwer}
We denote by $\Gamma(\omega; \mathbb{A})$ the set of solutions to the equation
$\Psi (X) = 0.$
We now show that $\Gamma(\omega; \mathbb{A}) \neq \emptyset$. Let $\mathbb{A} = (A_1, A_2, \ldots, A_m) \in \mathbb{P}^m_n$ with
$\alpha I \leq A_j \leq \beta I$ for all $j$ and for some $\alpha, \beta > 0$. Define the L\"owner interval
$K := [\alpha I, \beta I].$
For $X \in K$, we have
\[
(X^{1/2} A_i X^{1/2})^{1/2} \leq (\beta X)^{1/2} \leq \beta I.
\]
By the monotonicity property of the Karcher mean (P4), it follows that
\[
\Lambda \!\left( \omega; \, (X^{1/2} A_{1} X^{1/2})^{1/2}, (X^{1/2} A_{2} X^{1/2})^{1/2}, \ldots, (X^{1/2} A_{m} X^{1/2})^{1/2} \right) \leq \beta I.
\]
Similarly, one obtains
\[
\alpha I \leq
\Lambda \!\left( \omega; (X^{1/2} A_{1} X^{1/2})^{1/2}, (X^{1/2} A_{2} X^{1/2})^{1/2}, \ldots, (X^{1/2} A_{m} X^{1/2})^{1/2} \right).
\]
Thus, the continuous map $\Lambda \!\left( \omega; (X^{1/2} A_{1} X^{1/2})^{1/2}, (X^{1/2} A_{2} X^{1/2})^{1/2}, \ldots, (X^{1/2} A_{m} X^{1/2})^{1/2} \right)$ maps the compact convex set $K$ into itself. By Brouwer’s fixed point theorem, there exists a point $X \in K$ such that $\Phi(X) = 0$. We conclude that $\Gamma(\omega; \mathbb{A}) \neq \emptyset.$
\end{remark}

\begin{lemma}\label{ineq}
Let $\mathbb{A} = (A_1, A_2, \ldots, A_m) \in \mathbb{P}^{m}_n$ and $\omega = (w_{1}, w_{2}, \dots, w_{m}) \in \Delta_{m}$ with $\alpha I \leq A_j \leq \beta I $ for all $j$ and some $\alpha, \beta > 0$. Then $X \in \Gamma(\omega; \mathbb{A})$ satisfies
\begin{displaymath}
\alpha I \leq X \leq \beta I.
\end{displaymath}
\end{lemma}
\begin{proof}
Let $X \in \Gamma(\omega; \mathbb{A})$ . Since $\alpha I \leq A_{i} \leq \beta I$ for all $i$,
\begin{displaymath}
\sqrt{\alpha} X^{1/2} \leq (X^{1/2} A_{i} X^{1/2})^{1/2} \leq \sqrt{\beta} X^{1/2}.
\end{displaymath}
By the monotonicity of the Karcher mean with respect to the Loewner order (P4),
\begin{displaymath}
\sqrt{\alpha} X^{1/2} \leq \Lambda(\omega; (X^{1/2} A_{1} X^{1/2})^{1/2}, \dots, (X^{1/2} A_{m} X^{1/2})^{1/2}) = X \leq \sqrt{\beta} X^{1/2}.
\end{displaymath}
Solving the above for $X$, we obtain the desired inequalities.
\end{proof}

\begin{theorem} \label{T:-spe-epsilon}
Let $\omega = (w_{1}, w_{2}, \dots, w_{m}) \in \Delta_{m}$. Then there exists $\epsilon_{\omega} > 1$ such that for any $\mathbb{A} \in [\epsilon_{\omega}^{-1} I, \epsilon_{\omega} I]^m \subset  \mathbb{P}^{m}_n$, the equation $\Phi(X) =0$ has a unique solution in $[\epsilon_{\omega}^{-1} I, \epsilon_{\omega} I] \subset \mathbb{P}_n$.
\end{theorem}

\begin{proof}
We define the function $H^{\omega} : \mathbb{P}_n^m \times \mathbb{P}_n \rightarrow \mathbb{H}_n$ by
\begin{displaymath}
H^{\omega} (\mathbb{A},X)  := X - \Lambda(\omega; (X^{1/2} A_{1} X^{1/2})^{1/2},(X^{1/2} A_{2} X^{1/2})^{1/2}, \cdots,(X^{1/2} A_{m} X^{1/2})^{1/2}),
\end{displaymath}
for $(A_1, A_2, \ldots, A_m) \in \mathbb{P}_n^m$ and $\omega = (w_{1}, w_{2}, \dots, w_{m}) \in \Delta_{m}$. The function $H^{\omega}: \mathbb{P}_n^m \times \mathbb{P}_n \rightarrow \mathbb{H}_n$ is continuously Fréchet differentiable.
The derivative $D_2 H^{\omega}(\mathbb{A} ,X)$ maps from $\mathbb{H}_n$ to $\mathbb{H}_n$. Considering $\mathbb{I} = (I, \ldots, I) \in \mathbb{P}_n^m$, we find $H^{\omega}(\mathbb{I}, X) = X - X^{1/2}$ and $H^{\omega}(\mathbb{I}, I) = 0$.
The derivative $D_2 H^{\omega}(\mathbb{I}, X)(Y)$ is given by
\begin{displaymath}
D_2 H^{\omega}(\mathbb{I}, X)(Y) = Y - \int_0^\infty e^{-t X^{1/2}} Y e^{-t X^{1/2}} \, dt
\end{displaymath}
At $X = I$,
\begin{displaymath}
D_2 H^{\omega}(\mathbb{I}, I)(Y) = Y - \int_0^\infty e^{-2t} \, dt \, Y = Y - \frac{1}{2} Y = \frac{1}{2} Y.
\end{displaymath}
Thus, $D_2 H^{\omega}(\mathbb{I}, I) :\mathbb{H}_n \rightarrow \mathbb{H}_n $ is an invertible map.
By Implicit Function Theorem, there exist open sets $V \subseteq \mathbb{P}_n^m$ containing $\mathbb{I}$ and $W \subseteq \mathbb{P}_n$ containing $I$, and a continuously differentiable function $G : V \rightarrow W$ such that $H^{\omega}_{\mathbb{B}} ( G(\mathbb{B})) = 0$ for all $\mathbb{B} \in V$ and $G(\mathbb{I}) = I$.

Choose $\epsilon >1$ such that $[\epsilon^{-1} I , \epsilon I]^m \subset V \cap W$. Then $\mathbb{A} \in [\epsilon^{-1} I, \epsilon I]^m$.
For $X \in \Gamma(\omega, \mathbb{A} )$,
\begin{displaymath}
\epsilon^{-1} I \leq X \leq \epsilon I
\end{displaymath}
by Lemma \ref{ineq}.
Therefore, $X \in [\epsilon^{-1} I, \epsilon I] \subset W$. From the previous discussion, there exists a continuously differentiable function $G: [\epsilon^{-1} I , \epsilon I]^m \to [\epsilon^{-1} I , \epsilon I]$ such that $H^{\omega} (\mathbb{A},X)  = 0$ and $G(\mathbb{A}) = X$. That means $\Gamma(\omega; \mathbb{A}) = \{ X \}$.
Hence, the equation $\Phi(X) =0$ has a unique solution in $[\epsilon_{\omega}^{-1} I, \epsilon_{\omega} I]$.

\end{proof}

\begin{lemma}\label{lem:log-formula}
Let
\[
G = \begin{pmatrix} a & b \\[2pt] b & d \end{pmatrix} \in \mathbb{P}_2,
\qquad
\det G = 1.
\]
Set $\displaystyle M := \frac{G - G^{-1}}{2}$. Then $G = tI + M$ for $t = \frac{1}{2} \tr G$, and
% \[
% t:=\frac{a+d}{2},
% \qquad
% M:=
% \begin{pmatrix}
% \dfrac{a-d}{2} & b\\[6pt]
% b & -\dfrac{a-d}{2}
% \end{pmatrix}.
% \]
% Then \(t\ge 1\), \(G=tI+M\), \(M^2=(t^2-1)I\), and
\[
\log G = \theta \, M,
\]
where
\[
\theta =
\begin{cases}
\dfrac{\arccosh(t)}{\sqrt{t^2-1}}, & t>1,\\[10pt]
1, & t=1.
\end{cases}
\]
\end{lemma}

\begin{proof}
Since $G^{-1} = \begin{pmatrix} d & -b \\[2pt] -b & a \end{pmatrix}$, we obtain $M = \begin{pmatrix} \frac{a-d}{2} & b \\[2pt] b & -\frac{a-d}{2} \end{pmatrix}$ so $\displaystyle G = \left( \frac{a+d}{2} \right) I + M$.

Since \(G \in \mathbb{P}_2\) and \(\det G = 1\), its eigenvalues are \(\lambda, \lambda^{-1} > 0\), so
\[
t = \frac{\lambda + \lambda^{-1}}{2} \ge 1.
\]
Also,
\[
M^2 = \left( \frac{a-d}{2} \right)^2 I + b^2 I
= \left( \frac{(a+d)^2}{4} - ad + b^2 \right) I
= \left(t^2 - \det G \right) I
= (t^2-1) I.
\]
If \(t=1\), then \(\lambda=1\), hence \(G=I\), \(M=0\), and the formula is trivial.

Assume now that \(t > 1\).  Set
\[
U := \frac{M}{\sqrt{t^2-1}},
\]
so that \(U^2 = I\).  Writing \(\rho := \arccosh(t)\), we have $t = \cosh \rho$ and $\sqrt{t^2-1} = \sinh\rho$.
Therefore,
\[
G=tI+M
=
\cosh\rho\,I+\sinh\rho\,U
=
e^{\rho U}.
\]
Since \(G\in \PP_2\), the principal logarithm is
\[
\log G=\rho U
=
\frac{\arccosh(t)}{\sqrt{t^2-1}}\,M,
\]
as claimed.
\end{proof}
The Poincar\'e--Miranda theorem in dimension \(2\):
\begin{theorem}
 If \(F=(F_1,F_2)\) is continuous on a closed rectangle \([a,b] \times [c,d]\) such that
\[
F_1(a,v) \ge 0, \quad F_1(b,v) \le 0
\qquad (c \le v \le d)
\]
and
\[
F_2(u,c) \ge 0, \quad F_2(u,d) \le 0
\qquad (a \le u \le b),
\]
then there exists \((u_*,v_*) \in [a,b] \times [c,d]\) such that
\[
F_1(u_*,v_*) = F_2(u_*,v_*) = 0.
\]
\end{theorem}

Next, we give a counterexample that $\Phi(X)=0$ does not have a unique solution.

\begin{remark}\label{notuniq}
Let
$S = \begin{pmatrix} 1&0 \\[2pt] 0&-1 \end{pmatrix}, \
T = \begin{pmatrix} 0&1 \\[2pt] 1&0 \end{pmatrix}.$
Fix
$c := \frac{19}{10},$
$X_0 := e^{3S} = \begin{pmatrix} e^3&0 \\[2pt] 0&e^{-3} \end{pmatrix}.$
Define
\[
Z_1 := 3S, \qquad Z_2 := cT, \qquad Z_3 := -Z_{1}-Z_{2},
\]
and then set
\[
B_i := e^{Z_i}, \qquad A_i := B_i X_0 B_i \qquad (i=1,2,3).
\]
We consider $\displaystyle \Phi(X) := \frac{1}{3} \sum_{i=1}^3 \log(A_i \sharp X^{-1})$. Since each \(B_i\) is positive definite, the uniqueness of a positive definite solution to the Riccati equation gives
\[
A_i \sharp X_0^{-1} = B_i \qquad (i=1,2,3).
\]
Consequently,
\[
\Phi(X_0)
= \frac13 \sum_{i=1}^3 \log(A_i \sharp X_0^{-1})
= \frac13 \sum_{i=1}^3 \log(B_i)
= \frac13 \sum_{i=1}^3 Z_i
= 0.
\]
Thus, \(X_0\) is a solution.
We have
\[
\det A_i = \det(B_i)^2 \det(X_0) = 1 \qquad (i=1,2,3),
\]
because \(\det(B_i) = e^{\tr Z_i} = 1\) and \(\det(X_0) = 1\).
Observe that if \(X \in \PP_2\) is a solution of $\Phi(X) = 0$, then \(\det X = 1\).
Therefore, every solution of $\Phi(X)=0$ belongs to the manifold $\{X \in \PP_2: \det X = 1\}$.
Indeed, we can parametrize this manifold explicitly. For \(u,v\in \mathbb{R}\), set
\[
H(u,v) := uS + vT, \qquad
X(u,v) := e^{H(u,v)}.
\]
%Since \(H(u,v)^2=(u^2+v^2)I\), writing
%\[
%\rho:=\sqrt{u^2+v^2},
%\]
%we have
%\[
%X(u,v)
%=
%\cosh \rho\, I+\frac{\sinh\rho}{\rho}(uS+vT).
%\]
Then $H(u,v) \in \mathbb{H}_{2}, \ X(u,v) \in \PP_{2}$ and \(\det X(u,v) = e^{\tr H(u,v)} = 1\).
Conversely, if a real matrix \(X \in \PP_2\) satisfies \(\det X = 1\), then \(Y := \log X\) is real symmetric and $\tr Y = \log \det X = 0$.
Thus, \(Y = u_{0} S + v_{0} T\) for some \(u_{0}, v_{0} \in \mathbb{R}\), and therefore, \(X = X(u_{0}, v_{0})\). Hence every solution of \(\Phi(X) = 0\) can be written in the form \(X(u,v)\).

By \cite[Proposition 4.1.12]{Bh},
\[
G_i(u,v) := A_i \sharp X(u,v)^{-1}
= \frac{A_i + X(u,v)^{-1}}{\sqrt{\det(A_i + X(u,v)^{-1})}}
\qquad (i=1,2,3).
\]
Write
\[
G_i(u,v)=
\begin{pmatrix}
a_i(u,v) & b_i(u,v)\\[2pt]
b_i(u,v) & d_i(u,v)
\end{pmatrix},
\qquad i=1,2,3.
\]
Since \(G_i(u,v) \in \PP_2\) and \(\det G_i(u,v) = 1\), Lemma~\ref{lem:log-formula} yields
\[
\log G_i(u,v) = \theta_i(u,v)
\begin{pmatrix}
\dfrac{a_i(u,v)-d_i(u,v)}{2} & b_i(u,v)\\[8pt]
b_i(u,v) & -\dfrac{a_i(u,v)-d_i(u,v)}{2}
\end{pmatrix},
\]
where
\[
\theta_i(u,v)
= \frac{\arccosh(t_i(u,v))}{\sqrt{t_i(u,v)^2-1}},
\qquad
t_i(u,v) = \frac{a_i(u,v)+d_i(u,v)}{2},
\]
with the continuous interpretation \(\theta_i(u,v) = 1\) whenever \(G_i(u,v) = I\).
Equivalently,
\[
\log G_i(u,v)
= \theta_i(u,v) \left[ \frac{a_i(u,v)-d_i(u,v)}{2} S + b_i(u,v) T \right].
\]
Hence,
\[
\Phi \bigl( X(u,v) \bigr)
= F_1(u,v) S + F_2(u,v) T,
\]
where
\[
F_1(u,v) := \frac{1}{3} \sum_{i=1}^3 \theta_i(u,v) \, \frac{a_i(u,v)-d_i(u,v)}{2}
\quad \textrm{and} \quad
F_2(u,v) := \frac{1}{3} \sum_{i=1}^3 \theta_i(u,v) \, b_i(u,v).
\]
Therefore,
\[
\Phi \bigl( X(u,v) \bigr) = 0
\qquad \Longleftrightarrow \qquad
F_1(u,v) = F_2(u,v) = 0.
\]

Since $X_{0} = X(3,0)$, we now show that \(F_1(u,v) = F_2(u,v) = 0\) has another solution distinct from \((3,0)\).  The following numerical computation is given by MATLAB.
Let
\[
R = [u_-,u_+] \times [v_-,v_+],
\]
where
\[
u_-:=1.61247432825,
\qquad
u_+:=1.62347432825,
\]
and
\[
v_-:=0.5194188906,
\qquad
v_+:=0.5254188906.
\]
The following bounds hold:
\begin{align*}
F_1(u_-,v)&\ge 1.0174896754\times 10^{-4}>0
&&\text{for all } v\in[v_-,v_+],\\
F_1(u_+,v)&\le -9.2173535435\times 10^{-5}<0
&&\text{for all } v\in[v_-,v_+],\\
F_2(u,v_-)&\ge 9.9667442521\times 10^{-6}>0
&&\text{for all } u\in[u_-,u_+],\\
F_2(u,v_+)&\le -4.7891817896\times 10^{-6}<0
&&\text{for all } u\in[u_-,u_+].
\end{align*}
By Poincar\'e--Miranda theorem, there exists $(u_*,v_*) \in R$
such that $F_1(u_*,v_*) = F_2(u_*,v_*) = 0$.
Therefore,
\[
X_* := X(u_*,v_*) = e^{u_* S + v_* T}
\]
is another solution of \(\Phi(X) = 0\).
Since \((3,0) \notin R\), we have \(X_* \neq X_0\).  Therefore, the equation
$\Phi(X) = 0$ has at least two distinct solutions in \(\PP_2\).

Thus,
the local uniqueness theorem near the identity cannot be extended to a global uniqueness.
\end{remark}

\begin{remark}
The equation \eqref{phi} is introduced as an attempt to define a
multivariable spectral geometric mean. However, as shown above, in the multi-variable case the equation \eqref{phi} need not have a unique solution. Therefore, this construction
does not define a multivariable mean in full generality.
Nevertheless, the solutions of \eqref{phi} possess several
properties analogous to those of the two-variable spectral geometric mean.
These properties are studied in the next section.
\end{remark}

%\begin{remark}
%A numerical computation gives an approximate second solution inside the rectangle \(R\)
%\[
%(u_*,v_*) \approx (1.61797432825,\ 0.52241889060),
%\]
%which yields
%\[
%X_*\approx
%\begin{pmatrix}
%5.34715939657 & 0.81310419123\\[2pt]
%0.81310419123 & 0.31065810883
%\end{pmatrix}.
%\]
%\end{remark}

\begin{remark}[Necessity of the order-preserving assumption]\label{rem:order-preserving-necessary}
Gaubert and Qu \cite{GQ} characterized the exponential contraction rate of
order-preserving flows in Thompson's metric. In particular, for an
order-preserving flow of $\psi$ on a cone $\mathbb{P}$, the best contraction rate
is given by
\[
\alpha(\psi)
=
\sup\left\{
\alpha\in\mathbb{R}:
D\psi(x)[x]-\psi(x)\leq -\alpha x
\quad \text{for all } x\in \mathbb{P}
\right\}.
\]
Equivalently,
\[
\alpha(\psi)
=
\sup\left\{
\alpha\in\mathbb{R}:
D\left(X^{-1/2}\psi(X)X^{-1/2}\right)[X]
\leq -\alpha I
\quad \text{for all } X\in \mathbb{P}
\right\}.
\]
Consider
\[
\Psi(X)
=
\sum_{i=1}^{m}
w_i X^{-1/2}\log(A_i\sharp X^{-1})X^{-1/2}.
\]
For $F(X) := X^{-1/2} \Psi(X) X^{-1/2}$, we have
\[
DF(X)[X]=-\frac12 I .
\]
Thus, the above differential condition is satisfied with $\alpha=1/2$.

Nevertheless, $\Psi$ does not generate an order-preserving flow. Moreover, as
shown in Remark~\ref{notuniq}, the equation $\Psi(X)=0$ may have more than one
solution. Therefore, the flow generated by $\Psi$ cannot be exponentially
contractive.
This example shows that
the order-preserving property is necessary.
\end{remark}

\section{ Properties and Relationships with other least squares means }

Note that $\Gamma: \Delta_{m} \times \mathbb{P}_{n}^{m} \to \mathbb{P}_{n}$ defined by
\begin{displaymath}
\Gamma(\omega; \mathbb{A}) = \{ X \in \mathbb{P}_{n}: \Psi(X) = 0\}
\end{displaymath}
is a set-valued function satisfying that $\Gamma(\omega; \mathbb{A}) \neq \emptyset$ from  Remark \ref{brouwer}. Indeed,
\begin{equation} \label{E:set-valued}
\begin{split}
X \in \Gamma(\omega; \mathbb{A}) & \quad \Leftrightarrow \quad \Lambda(\omega; A_{1} \sharp X^{-1}, \dots, A_{m} \sharp X^{-1}) = I \\
& \quad \Leftrightarrow \quad \Lambda(\omega; A_{1}^{-1} \sharp X, \dots, A_{m}^{-1} \sharp X) = I
\end{split}
\end{equation}
by self-duality of the Karcher mean in (P8). 
Moreover, $\Gamma(1-t,t; A, B) = \{ A \natural_{t} B \}$ for $t \in [0,1]$ and $A, B \in \mathbb{P}_{n}$.

For given $\omega = (w_1, \dots, w_m) \in \Delta_{m}$ and $\mathbb{A}= (A_1, \ldots A_m) \in \mathbb{P}_n^m$ we denote
\begin{displaymath}
\begin{split}
\omega^k & := \frac{1}{k} (w_1, \dots, w_m, \cdots, w_1, \dots, w_m) \in \Delta_{mk}, \\
\mathbb{A}^{k} & := (A_{1}, \dots, A_{m}, \cdots, A_{1}, \dots, A_{m}) \in \mathbb{P}_{n}^{mk}
\end{split}
\end{displaymath}
for any $k \in \mathbb{N}$, and
\begin{displaymath}
c \Gamma(\omega; \mathbb{A}) := \{ c X: X \in \Gamma(\omega; \mathbb{A}) \}, \quad
S \Gamma(\omega; \mathbb{A}) S^{*} := \{ S X S^{*}: X \in \Gamma(\omega; \mathbb{A}) \}
\end{displaymath}
for any $c > 0$ and invertible matrix $S$. 
Applying the properties of Karcher mean in Theorem \ref{T:properties} to \eqref{E:set-valued}, we obtain

\begin{theorem} \label{T:multi-sp-geomean}
Let $\omega = (w_1, \dots, w_m) \in \Delta_{m}$ and $\mathbb{A}= (A_1, \ldots A_m) \in \mathbb{P}_n^m$.
\begin{enumerate}
\item $\Gamma(\omega ; \mathbb{A}) \supseteq \{ A_1^{w_1} A_2^{w_2} \cdots A_m^{w_m} \}$ if $A_i$'s commute;

\item $\Gamma(\omega ; \alpha \mathbb{A}) = \alpha \Gamma(\omega ; \mathbb{A})$ for any $\alpha > 0$;

\item $\Gamma(\omega_\sigma ; \mathbb{A}_{\sigma}) = \Gamma(\omega ; \mathbb{A})$ for any permutation $\sigma$ on $m$-letters;

\item $\Gamma(\omega^k ; \mathbb{A}^k )= \Gamma(\omega ; \mathbb{A})$ for any $k \in \mathbb{N}$;

\item $\Gamma(\omega ; U \mathbb{A} U^*) = U \Gamma(\omega ; \mathbb{A}) U^*$ for any unitary matrix $U$;

\item $  \Gamma(\omega ; \mathbb{A}^{-1}) := \{ X^{-1}: X \in \Gamma(\omega; \mathbb{A}) \}. $

\item $\det \Gamma(\omega ; \mathbb{A}) = \{ (\det A_{1})^{w_{1}} (\det A_{2})^{w_{2}} \cdots (\det A_{m})^{w_{m}} \}$, where $$ \det \Gamma(\omega ; \mathbb{A}) := \{ \det X: X \in \Gamma(\omega; \mathbb{A}) \}. $$
\end{enumerate}
\end{theorem}

% \begin{proof}
% Properties (3) and (4) follow from the definition of $\Gamma(\omega ; \mathbb{A})$.
% \begin{itemize}
%   \item[(1)] Assume that $A_{i}$'s commute. 
% Then $Z := A_1^{w_1} \cdots A_m^{w_m}$ also commutes with each $A_{i}$. 
% By consistency with scalars of the Karcher mean in (P1)
% \begin{displaymath}
% \Lambda(\omega; A_{1}^{-1} \sharp Z, \dots, A_{m}^{-1} \sharp Z) 
% = \prod_{j=1}^{m} (A_{j}^{-1} \sharp Z)^{w_{j}} 
% = Z^{1/2} \prod_{j=1}^{m} A_{j}^{-w_{j}/2} = I.
% \end{displaymath}
% Thus, $Z \in \Gamma(\omega ; \mathbb{A})$.

%   \item[(2)] Let $X \in \Gamma(\omega ; \mathbb{A})$. 
% Since $A_{i} \sharp X^{-1} = (\alpha A_{i}) \sharp (\alpha X)^{-1}$ for any $\alpha > 0$,
% \begin{displaymath}
% I = \Lambda(\omega; A_{1} \sharp X^{-1}, \dots, A_{m} \sharp X^{-1}) 
% = \Lambda(\omega; (\alpha A_{1}) \sharp (\alpha X)^{-1}, \dots, (\alpha A_{m}) \sharp (\alpha X)^{-1}).
% \end{displaymath}
% This means that $\alpha X \in \Gamma(\omega ; \alpha \mathbb{A})$. Conversely, let $Y \in \Gamma(\omega ; \alpha \mathbb{A})$. Then
% \begin{displaymath}
% I = \Lambda(\omega; (\alpha A_{1}) \sharp Y^{-1}, \dots, (\alpha A_{m}) \sharp Y^{-1}) 
% = \Lambda(\omega; A_{1} \sharp (\alpha^{-1} Y)^{-1}, \dots, A_{m} \sharp (\alpha^{-1} Y)^{-1}).
% \end{displaymath}
% This means that $\alpha^{-1} Y \in \Gamma(\omega ; \mathbb{A})$, that is, $Y \in \alpha \Gamma(\omega ; \mathbb{A})$.
% \end{itemize}
% \end{proof}

\begin{remark}\label{R:spectral}
For the $k$-th antisymmetric tensor power $\wedge^{k}$, note that
\begin{displaymath}
\wedge^{k}\Gamma(\omega; \mathbb{A}) := \{\wedge^{k}X: X \in \Gamma(\omega;\mathbb{A})\} \subset \Gamma(\omega; \wedge^{k}\mathbb{A}),
\end{displaymath}
where $\wedge^{k}\mathbb{A} := (\wedge^{k} A_{1}, \dots, \wedge^{k} A_{m})$.
% Note that $\mathcal{S}(\omega; \mathbb{A})$ is preserved by the $k$-th antisymmetric tensor power $\wedge^{k}$.
Indeed, let $X = \mathcal{S}(\omega; \mathbb{A})$. Taking $\wedge^{k}$ to both sides of \eqref{E:n-equation}
\begin{displaymath}
\begin{split}
I = \wedge^{k} I & = \wedge^{k} \Lambda(\omega; X^{-1} \sharp A_{1}, \dots, X^{-1} \sharp A_{m}) \\
& = \Lambda(\omega; \wedge^{k} (X^{-1} \sharp A_{1}), \dots, \wedge^{k} (X^{-1} \sharp A_{m})) \\
& = \Lambda(\omega; (\wedge^{k} X)^{-1} \sharp (\wedge^{k} A_{1}), \dots, (\wedge^{k} X)^{-1} \sharp (\wedge^{k} A_{m})).
\end{split}
\end{displaymath}
So $\wedge^{k} X \in \Gamma(\omega; \wedge^{k}\mathbb{A})$.
See more details and properties about the antisymmetric tensor power in \cite[Section 2]{JK}.
\end{remark}

In the following we denote as $\mathcal{S}(\omega ; \mathbb{A})$ arbitrary element in $\Gamma(\omega; \mathbb{A})$ and show its properties. 
\begin{theorem} \label{T:bounds}
Let $\mathbb{A} = (A_{1}, \dots, A_{m}) \in \mathbb{P}_{n}^{m}$ and $\omega \in \Delta_{m}$. Then
\begin{displaymath}
2I - \sum_{i=1}^{m} w_{i} A_{i}^{-1} \leq \mathcal{S}(\omega; \mathbb{A}) \leq \left[ 2I - \sum_{i=1}^{m} w_{i} A_{i} \right]^{-1}.
\end{displaymath}
\end{theorem}

\begin{proof}
Let $X = \mathcal{S}(\omega; \mathbb{A})$. By the weighted arithmetic-$\Lambda$-harmonic mean inequalities (P10)
\begin{displaymath}
\left[ \sum_{i=1}^{m} w_{i} (X \sharp A_{i}^{-1}) \right]^{-1} \leq \Lambda(\omega; X^{-1} \sharp A_{1}, \dots, X^{-1} \sharp A_{m}) = I \leq \sum_{i=1}^{m} w_{i} (X^{-1} \sharp A_{i}).
\end{displaymath}
By using the two-variable arithmetic-geometric mean inequality, the above inequalities imply
\begin{displaymath}
\left[ \frac{X + \sum_{i=1}^{m} w_{i} A_{i}^{-1}}{2} \right]^{-1} \leq I \leq \frac{X^{-1} + \sum_{i=1}^{m} w_{i} A_{i}}{2}.
\end{displaymath}
Solving the above for $X$, we obtain the desired inequalities.
\end{proof}

The bounds in Theorem \ref{T:bounds} are the same as those of Wasserstein mean in \cite[Theorem 3.4]{HK19}. Thus, every element in $\Gamma(\omega; \mathbb{A})$ satisfies the extended Lie-Trotter formula.
\begin{theorem}
Let $\mathbb{A} = (A_{1}, \dots, A_{m}) \in \mathbb{P}_{n}^{m}$ and $\omega \in \Delta_{m}$. Then
\begin{displaymath}
\lim_{s \to 0} \mathcal{S}(\omega; \mathbb{A}^{s})^{1/s} = \mathrm{LE}(\omega; \mathbb{A}).
\end{displaymath}
\end{theorem}

\begin{lemma}\label{lem:two-sided-log-majorization}
Let \( G: \Delta_{m} \times \mathbb{P}_{n}^{m} \to \mathbb{P}_{n} \) be a matrix mean satisfying the homogeneity and invariance under the $k$-th antisymmetric tensor power for all $k \geq 1$.
Let \( \Theta(\omega;\mathbb{A})\subset \mathbb{P}_{n} \) be a nonempty set satisfying
\begin{equation} \label{E:homogeneity}
\Theta(\omega; c \mathbb{A}) = c \, \Theta(\omega;\mathbb{A})
:= \{ cY: Y \in \Theta(\omega; \mathbb{A}) \},
\end{equation}
and
\begin{equation} \label{E:antisym-tensor-power}
\wedge^{k} \Theta(\omega; \mathbb{A})
:= \{ \wedge^{k}Y : Y \in \Theta(\omega; \mathbb{A}) \}
\subset \Theta(\omega; \wedge^{k} \mathbb{A}).
\end{equation}
Then the following statements hold:
\begin{enumerate}
\item If $X \le I$ for every \( X \in \Theta(\omega; \mathbb{A}) \) implies $G(\omega; \mathbb{A}) \le I$, then $G(\omega;\mathbb{A}) \prec_{w \log} X$.

\item If $G(\omega; \mathbb{A}) \le I$ implies $X \le I$ for every \( X \in \Theta(\omega;\mathbb{A}) \), then $X \prec_{w \log} G(\omega; \mathbb{A})$.
\end{enumerate}
\end{lemma}

\begin{proof}
We first prove (1). Let \(\lambda_{1}(X) \ge \cdots \ge \lambda_{n}(X)\) denote the eigenvalues of \(X \in \mathbb{P}_{n}\).
Fix \(1 \le k \le n\), and define
\[
\mu_{k} := \lambda_{1}(\wedge^{k} X) = \prod_{j=1}^{k} \lambda_{j}(X).
\]
Since \(\wedge^{k} X \in \Theta(\omega; \wedge^{k} \mathbb{A})\) for every \( X \in \Theta(\omega; \mathbb{A}) \) by \eqref{E:antisym-tensor-power} and
\(\Theta\) is homogeneous by \eqref{E:homogeneity}, we have
\[
\mu_{k}^{-1} (\wedge^{k} X)
\in
\Theta(\omega; \mu_{k}^{-1} (\wedge^{k} \mathbb{A})).
\]
Note that $\mu_{k}^{-1} (\wedge^{k} X) \le I$ by the definition of \(\mu_{k}\).
Applying the hypothesis in (1) and homogeneity of \( G \), we obtain
$\mu_{k}^{-1} G(\omega; \wedge^{k} \mathbb{A}) = G(\omega; \mu_{k}^{-1} (\wedge^{k} \mathbb{A})) \le I$.
So $G(\omega; \wedge^{k} \mathbb{A}) \le \mu_{k} I$.
Since $G$ is invariant under the $k$-th antisymmetric tensor power,
\[
\wedge^{k} G(\omega; \mathbb{A}) = G(\omega; \wedge^{k} \mathbb{A}) \le \mu_{k} I,
\]
which implies $\displaystyle \lambda_{1} \bigl(\wedge^{k} G(\omega; \mathbb{A}) \bigr) = \prod_{j=1}^{k} \lambda_{j} \bigl( G(\omega; \mathbb{A}) \bigr) \le \mu_{k}$,
that is, $G(\omega; \mathbb{A}) \prec_{w \log} X$.

Next we prove (2). Fix \(1 \le k \le n\), and define
\[
\nu_{k} := \lambda_{1} \bigl( \wedge^{k} G(\omega; \mathbb{A}) \bigr)
= \prod_{j=1}^{k} \lambda_{j} \bigl( G(\omega; \mathbb{A}) \bigr).
\]
By the homogeneity and \(\wedge^{k}\)-invariance of \( G \),
\[
G(\omega; \nu_{k}^{-1} (\wedge^{k} \mathbb{A}))
= \nu_{k}^{-1} G(\omega; \wedge^{k} \mathbb{A})
= \nu_{k}^{-1} \bigl( \wedge^{k} G(\omega; \mathbb{A}) \bigr)
\le I.
\]
Since \( \wedge^{k} X \in \Theta(\omega; \wedge^{k} \mathbb{A})\) for every \( X \in \Theta(\omega; \mathbb{A}) \) by \eqref{E:antisym-tensor-power} and
\(\Theta\) is homogeneous by \eqref{E:homogeneity}, we have
$\nu_{k}^{-1} (\wedge^{k} X) \in \Theta(\omega; \nu_{k}^{-1} (\wedge^{k} \mathbb{A}))$.
Applying the hypothesis in (2), we obtain $\wedge^{k} X \le \nu_{k} I$, which implies
\[
\lambda_{1} (\wedge^{k} X) \le \nu_{k},
\]
that is, $\displaystyle \prod_{j=1}^{k} \lambda_{j}(X) \le
\prod_{j=1}^{k} \lambda_{j} \bigl( G(\omega; \mathbb{A}) \bigr)$ for all $1 \le k \le n$.
Therefore, $X \prec_{w \log} G(\omega; \mathbb{A})$.
\end{proof}
% From Theorem \ref{T:-spe-epsilon} we have seen that for $\mathbb{A} = (A_{1}, \dots, A_{m}) \in [\epsilon_{\omega}^{-1} I, \epsilon_{\omega} I]^{m}$ for some $\epsilon_{\omega} > 1$, the equation $\Phi(X) = 0$ has a unique positive definite solution, denoted by $\mathcal{S}_{\epsilon}(\omega; \mathbb{A})$. We verify its properties with other least squares means.

Note from \cite{GT24} that $e^{(1-t) \log A + t \log B} \prec_{\log} A \natural_{t} B$ for any $t \in [0,1]$. The following is a multi-variable extension of this result.
\begin{theorem} \label{T:LE-Sp}
Let $\mathbb{A} = (A_{1}, \dots, A_{n}) \in \mathbb{P}_{n}^{m}$ and $\omega \in \Delta_{m}$. Then
\begin{displaymath}
\mathrm{LE}(\omega; \mathbb{A}) \prec_{\log} \mathcal{S}(\omega; \mathbb{A}).
\end{displaymath}
\end{theorem}

\begin{proof}
We first show that
\begin{center}
$\mathcal{S}(\omega; \mathbb{A}) \leq I$ \quad implies \quad $\mathrm{LE}(\omega; \mathbb{A}) \leq I$.
\end{center}
Assume that $X = \mathcal{S}(\omega; \mathbb{A}) \leq I$. Then $X^{-1} \geq I$, and $X^{-1} \sharp A_{i} \geq I \sharp A_{i} = A_{i}^{1/2}$ for all $i$ by the monotonicity of geometric mean with respect to the Loewner order. So
\begin{displaymath}
(X^{-1} \sharp A_{i})^{p} \geq A_{i}^{p/2}
\end{displaymath}
for $0 < p \leq 1$ by the monotonicity of the $p$-power map. Then
\begin{displaymath}
\begin{split}
I = \Lambda(\omega; X^{-1} \sharp A_{1}, \dots, X^{-1} \sharp A_{n}) & = \Lambda(\omega; (X^{-1} \sharp A_{1})^{p}, \dots, (X^{-1} \sharp A_{n})^{p}) \\
& \geq \Lambda(\omega; A_{1}^{p/2}, \dots, A_{n}^{p/2}),
\end{split}
\end{displaymath}
where the second equality follows from the unique positive definite solution of Karcher equation and the inequality holds due to the monotonicity of Karcher mean.
Hence,
\begin{displaymath}
\Lambda(\omega; A_{1}^{p/2}, \dots, A_{n}^{p/2})^{2/p} \leq I.
\end{displaymath}
Taking limit as $p \to 0^{+}$ we get $\mathrm{LE}(\omega; \mathbb{A}) \leq I$.

From Remark \ref{R:spectral} and Lemma \ref{lem:two-sided-log-majorization} we obtain $\mathrm{LE}(\omega; \mathbb{A}) \prec_{w \log} \mathcal{S}(\omega; \mathbb{A})$.
Since
\begin{displaymath}
\det \mathrm{LE}(\omega; \mathbb{A}) = \exp \left[ \sum_{i=1}^{n} w_{i} \tr (\log A_{i}) \right] = \prod_{i=1}^{n} (\det A_{i})^{w_{i}} = \det \mathcal{S}(\omega; \mathbb{A})
\end{displaymath}
by Theorem \ref{T:multi-sp-geomean} (6), we complete the proof.
\end{proof}

The following gives us bounds for the solution set $\Gamma(\omega; \mathbb{A})$ with respect to near order.
\begin{theorem} \label{T:A-Sp-H}
Let $\mathbb{A} = (A_{1}, \dots, A_{m}) \in \mathbb{P}_{n}^{m}$ and $\omega \in \Delta_{m}$. Then
\begin{displaymath}
\mathcal{H}(\omega; \mathbb{A}) \preceq \mathcal{S}(\omega; \mathbb{A}) \preceq \mathcal{A}(\omega; \mathbb{A}).
\end{displaymath}
\end{theorem}

\begin{proof}
Let $X = \mathcal{S}(\omega; \mathbb{A})$. Then by joint concavity of the two-variable geometric mean and the weighted arithmetic-Karcher mean inequality in (P10)
\begin{displaymath}
\left( \sum_{j=1}^{m} w_{j} A_{j}^{-1} \right) \sharp X \geq \sum_{j=1}^{m} w_{j} (A_{j}^{-1} \sharp X) \geq \Lambda(\omega; A_{1}^{-1} \sharp X, \dots, A_{m}^{-1} \sharp X) = I.
\end{displaymath}
The last equality follows from \eqref{E:set-valued}. So $\displaystyle \left( \sum_{j=1}^{m} w_{j} A_{j}^{-1} \right)^{-1} \preceq X$. 
Similarly, 
\begin{displaymath}
X^{-1} \sharp \left( \sum_{j=1}^{m} w_{j} A_{j} \right) \geq \sum_{j=1}^{m} w_{j} (X^{-1} \sharp A_{j}) \geq \Lambda(\omega; X^{-1} \sharp A_{1}, \dots, X^{-1} \sharp A_{1}) = I.
\end{displaymath}
Thus, $\displaystyle X \preceq \sum_{j=1}^{m} w_{j} A_{j}$.
\end{proof}

\begin{remark}
For $X, Y \in \mathbb{H}_{n}$, $X \leq_{\lambda} Y$ if and only if $\lambda_{j}(X) \leq \lambda_{j}(Y)$ for all $j = 1, \dots, n$, where eigenvalues of $X, Y$ are rearranged in decreasing order: $\lambda_{1}(X) \geq \cdots \geq \lambda_{n}(X)$. We call $\leq_{\lambda}$ the pointwise eigenvalue order.
Since $A \leq B \Rightarrow A \preceq B \Rightarrow A \leq_{\lambda} B \Rightarrow A \prec_{w \log} B$ for $A, B \in \mathbb{P}_{n}$, Theorem \ref{T:A-Sp-H} provides the same boundedness for $\mathcal{S}(\omega; \mathbb{A})$ with respect to the pointwise eigenvalue order and weak log-majorization.
\end{remark}

% \begin{proof}
% Since $\mathcal{H}(\omega; \mathbb{A}) \leq \Lambda(\omega; \mathbb{A}) \prec_{\log} \mathrm{LE}(\omega; \mathbb{A})$, the first inequality follows from Theorem \ref{T:LE-Sp}.
% Since the weighted arithmetic mean is homogeneous, it is enough to prove that
% \begin{center}
% $\mathcal{A}(\omega; \mathbb{A}) \leq I$ \quad implies \quad $\mathcal{S}(\omega; \mathbb{A}) \leq I$.
% \end{center}
% If $\mathcal{A}(\omega; \mathbb{A}) \leq I$ then $\displaystyle 2I - \sum_{i=1}^{m} w_{i} A_{i} \geq I$, and thus, by Theorem \ref{T:bounds}
% \begin{displaymath}
% \mathcal{S}(\omega; \mathbb{A}) \leq \left[ 2I - \sum_{i=1}^{m} w_{i} A_{i} \right]^{-1} \leq I.
% \end{displaymath}
% \end{proof}

\begin{proposition}
Let $\mathbb{A} = (A_{1}, \dots, A_{m}) \in \mathbb{P}_{n}^{m}$ and $\omega \in \Delta_{m}$. If $\mathcal{S}(\omega; \mathbb{A}) \leq I$ then
\begin{displaymath}
\tr \mathcal{S}(\omega; \mathbb{A})^{3} \leq \tr \Omega(\omega; \mathbb{A}).
\end{displaymath}
\end{proposition}

\begin{proof}
Assume $X = \mathcal{S}(\omega; \mathbb{A}) \leq I$. Then by the weighted arithmetic-Karcher mean inequality (P10)
\begin{displaymath}
X = \Lambda(\omega; (X^{1/2} A_{1} X^{1/2})^{1/2}, \dots, (X^{1/2} A_{m} X^{1/2})^{1/2}) \leq \sum_{i=1}^{m} w_{i} (X^{1/2} A_{i} X^{1/2})^{1/2}.
\end{displaymath}
By the monotonicity of a square map with respect to the near order,
\begin{displaymath}
X^{2} \preceq \left[ \sum_{i=1}^{m} w_{i} (X^{1/2} A_{i} X^{1/2})^{1/2} \right]^{2}.
\end{displaymath}
Since $X \leq I$,
\begin{displaymath}
X^{3} \preceq X^{-1/2} \left[ \sum_{i=1}^{m} w_{i} (X^{1/2} A_{i} X^{1/2})^{1/2} \right]^{2} X^{-1/2} = K(X).
\end{displaymath}
By Lemma \ref{L:Wass-iteration} we conclude $\tr X^{3} \leq \tr K(X) \leq \tr \Omega(\omega; \mathbb{A})$.
\end{proof}

The Lim--P\'alfia power mean \cite{LP} is a non-commutative analogue of the classical scalar power mean. Let $\mathbb{A} = (A_{1}, \dots, A_{m}) \in \mathbb{P}_{n}^{m}$ and $\omega = (w_{1}, \dots, w_{m}) \in \Delta_{m}$.
The Lim--P\'alfia power mean $P_t(\omega; \mathbb{A})$ of order $t \in (0,1]$ is defined as the unique positive definite solution \( X \) of the equation
\[
X = \sum_{i=1}^m w_i \left( X \sharp_t A_i \right).
\]
For negative orders,
\[
P_t(\omega; \mathbb{A}) = P_{-t}(\omega; \mathbb{A}^{-1})^{-1},
\qquad t \in [-1,0).
\]
The remarkable result of Lim--P\'alfia power mean is that it converges to the Karcher mean monotonically with respect to the Loewner order \cite{LL14}:
\begin{equation} \label{E:Pt-monotone}
\mathcal{H} = P_{-1} \leq P_{-t} \leq P_{-s} \leq P_{0} = \Lambda \leq P_{s} \leq P_{t} \leq P_{1} = \mathcal{A}, \quad 0 < s \leq t \leq 1
\end{equation}
where $\displaystyle P_{0}(\omega; \mathbb{A}) := \lim_{t \to 0} P_{t}(\omega; \mathbb{A})$.

The following provides the relationship between $\mathcal{S}(\omega; \mathbb{A})$ and Lim-P\'alfia's power mean.
\begin{theorem} \label{T:Sp-Power}
Let $\mathbb{A} = (A_{1}, \dots, A_{m}) \in \mathbb{P}_{n}^{m}$ and $\omega \in \Delta_{m}$. If $\mathcal{S}(\omega; \mathbb{A}) \geq I$, then
\begin{displaymath}
\mathcal{S}(\omega; \mathbb{A})^{-1} \leq P_{t}(\omega; \mathbb{A}), \quad t \in [1/2,1].
\end{displaymath}
\end{theorem}

\begin{proof}
Assume that $\mathcal{S}(\omega; \mathbb{A}) \geq I$. Then $X := \mathcal{S}(\omega; \mathbb{A})^{-1} \leq I$, and
\begin{displaymath}
X \leq I = \Lambda(\omega; X \sharp A_{1}, \dots, X \sharp A_{m}) \leq \sum_{i=1}^{m} w_{i} (X \sharp A_{i}) =: f(X)
\end{displaymath}
by the weighted arithmetic-Karcher mean inequality. Since the map $f$ is operator monotone, $X \leq f(X) \leq f^{2}(X) \leq \cdots \leq f^{k}(X)$. Since $f^{k}(X)$ converges to $P_{1/2}(\omega; \mathbb{A})$ as $k \to \infty$, we obtain
\begin{displaymath}
X \leq P_{1/2}(\omega; \mathbb{A}).
\end{displaymath}
Since Lim-Palfia's power mean is monotone on parameter by \eqref{E:Pt-monotone}, we obtain the desired result.
% \begin{displaymath}
% \mathcal{S}(\omega; \mathbb{A}^{-1}) = \mathcal{S}(\omega; \mathbb{A})^{-1} \leq P_{t}(\omega; \mathbb{A}), \quad t \in [1/2,1]
% \end{displaymath}
% when $\mathcal{S}(\omega; \mathbb{A}) \leq I$.
\end{proof}

\noindent \textbf{Open problems:}
\begin{itemize}
\item[(1)] From Theorem \ref{T:LE-Sp} we have shown
\begin{displaymath}
\Lambda(\omega; \mathbb{A}) \prec_{\log} \mathrm{LE}(\omega; \mathbb{A}) \prec_{\log} \mathcal{S}(\omega; \mathbb{A}).
\end{displaymath}
As a multi-variable extension of $A \natural_{t} B \prec_{w \log} A \diamond_{t} B$, it is an open question to prove
\begin{displaymath}
\mathcal{S}(\omega; \mathbb{A}) \prec_{w \log} \Omega (\omega; \mathbb{A}).
\end{displaymath}
The near-order inequality in Theorem \ref{T:A-Sp-H} gives us an affirmative answer for the second weak log-majorization.

\item[(2)] Gan and Tam \cite{GT24} have shown
\begin{displaymath}
(A^{p} \natural_{t} B^{p})^{1/p} \prec_{\log} A \natural_{t} B, \quad 0 < p < 1.
\end{displaymath}
It implies that $(A^{p} \natural_{t} B^{p})^{1/p}$ converges to the log-Euclidean mean decreasingly with respect to the log-majorization. As a multi-variable extension of the preceding result, it is an interesting question whether the following holds:
\begin{displaymath}
\mathcal{S}(\omega; \mathbb{A}^{p})^{1/p} \prec_{\log} \mathcal{S} (\omega; \mathbb{A}), \quad 0 < p < 1.
\end{displaymath}
\end{itemize}

\section{Multivariable alternative means}

P\'alfia \cite{P} generalized the setting of the Karcher equation~\eqref{Karcher-eq} by replacing the specific function \( f(x) = \log(x) \) with an arbitrary operator monotone function \( f \) and by extending the framework from finite sums over \( m \)-tuples of operators to integrals with respect to probability measures supported on the cone \( \mathbb{P} \) of positive definite operators. This generalization extends a large part of the Kubo-Ando theory to the multivariable setting. In this framework, each operator mean can be characterized extrinsically as a unique solution of a generalized Karcher equation of the form:
\begin{equation} \label{generalized Karcher eq}
 \int_{\mathbb{P}} g(X^{-1/2} A X^{-1/2}) \, d\sigma(A) = 0,
\end{equation}
where \( g \in \mathcal{L} := \left\{ g : (0, \infty) \rightarrow \mathbb{R} \, \middle| \, g \text{ is operator monotone},\, g(1) = 0,\, g'(1) = 1 \right\} \), and \( \sigma \) is a Borel probability measure on \( \mathbb{P} \). A unique solution \( \Lambda_f(\sigma) \in \mathbb{P} \) is referred to as a \emph{generalized Karcher mean}.

A systematic framework has recently been developed to study \emph{alternative means} for positive definite operators, which include the weighted spectral geometric mean and the Wasserstein mean. The alternative mean associated to a normalized
operator monotone function $f$ is defined in \cite{DFKC} as
\begin{align*}
 A \hat{\sigma}_f B &= (I \sigma_f (A^{-1} \sharp B))A (I \sigma_f (A^{-1} \sharp B)) \qquad A, B \in \mathbb{P} \\
 &= f (A^{-1} \sharp B) A f (A^{-1} \sharp B).
 \end{align*}
We prove the following theorem for alternative means.

\begin{theorem}
Let \(A,B\in \mathbb{P}\), and let \(\sigma_f\) be an operator mean with strictly monotone representing function \(f\). Then \(A\hat{\sigma}_f B\) is the unique matrix \(X\in \mathbb{P}\) satisfying
\begin{equation}\label{alteq}
(A\sharp X^{-1}) \sigma_f (B\sharp X^{-1}) = I.
\end{equation}
\end{theorem}
\begin{proof}
Let $U = A \sharp X^{-1}, \ V = B \sharp X^{-1}$.
Then by assumption, \( U \sigma_f V = I \). By the definition of the Kubo-Ando operator mean \( \sigma_f \), this implies
\[
V  = U^{1/2} f^{-1}(U^{-1}) U^{1/2}.
\]
Now, we express \( X \) in terms of \( U \) and \( A \) by the Riccati equation:
$X = U^{-1} A U^{-1}$, and similarly,
$X = V^{-1} B V^{-1}$.
Thus, equating both expressions for \( X \), we get
\[
U^{-1} A U^{-1} = V^{-1} B V^{-1}.
\]
Substituting the expression for \( V \) in terms of \( U \), we have
\begin{align*}
A &= U V^{-1} B V^{-1} U \\
  &= U \left( U^{1/2} f^{-1}(U^{-1}) U^{1/2} \right)^{-1} B \left( U^{1/2} f^{-1}(U^{-1}) U^{1/2} \right)^{-1} U \\
  &= U^{1/2} \left[ f^{-1}(U^{-1}) \right]^{-1} U^{-1/2} B U^{-1/2} \left[ f^{-1}(U^{-1}) \right]^{-1} U^{1/2} \\
  &= \left[ f^{-1}(U^{-1}) \right]^{-1} B \left[ f^{-1}(U^{-1}) \right]^{-1}.
\end{align*}
This implies $[f^{-1}(U^{-1})]^{-1} = A \sharp B^{-1}$
so $f^{-1}(U^{-1}) = A^{-1} \sharp B$.
Applying \( f \) to both sides, we obtain
\[
U^{-1} = f(A^{-1} \sharp B).
\]
Finally, substituting this back into the expression for \( X \) yields
$$ X = U^{-1} A U^{-1} = f(A^{-1} \sharp B) \, A \, f(A^{-1} \sharp B) = A \hat{\sigma}_f B.$$
This shows the equation \eqref{alteq} has a unique solution.
\end{proof}

Now we consider the following equation  for positive definite matrices to define multi-variable alternative means:
\begin{align} \label{phi_g}
\Phi_g(X)
&= \Lambda_g \!\big( \omega; A_{1} \sharp X^{-1}, A_{2} \sharp X^{-1}, \ldots, A_{m} \sharp X^{-1} \big) - I = 0,
\end{align}
where $g \in \mathcal{L}$. Note that \eqref{phi_g} is a multi-variable extension of \eqref{alteq}, and is equivalent to
\[
 \Lambda_g \!\left(\omega; (X^{1/2} A_{1} X^{1/2})^{1/2}, (X^{1/2} A_{2} X^{1/2})^{1/2}, \ldots, (X^{1/2} A_{m} X^{1/2})^{1/2} \right) = X,
\]
by the congruence invariance of the generalized Karcher mean.
Moreover, from the generalized Karcher equation \eqref{generalized Karcher eq}, it can be expressed equivalently as
\[
\sum_{i=1}^{m} w_i \, g(A_i \sharp X^{-1}) = 0 \qquad \text{for some } g \in \mathcal{L}.
\]
We denote by $\Gamma_g(\omega; \mathbb{A})$ the set of solutions to the equation $\Phi_g(X) = 0$.
By an argument similar to that in Remark~\ref{brouwer}, one can show that $\Gamma_g(\omega; \mathbb{A}) \neq \emptyset$.

\begin{lemma} \label{ineq_1}
Let $\mathbb{A} = (A_1, A_2, \ldots, A_m) \in \mathbb{P}_n^{m}$ and $\omega = (w_{1}, w_{2}, \dots, w_{m}) \in \Delta_{m}$ with $\alpha I \leq A_j \leq \beta I $ for all $j$ and some $\alpha, \beta > 0$. Then $X \in \Gamma_g(\omega; \mathbb{A})$ satisfies
\begin{displaymath}
\alpha I \leq X \leq \beta I.
\end{displaymath}
\end{lemma}

 \begin{theorem}  \label{T:alt}
Let $\omega = (w_{1}, w_{2}, \dots, w_{m}) \in \Delta_{m}$. Then there exists $\epsilon_{\omega} > 1$ such that for any $\mathbb{A} \in [\epsilon_{\omega}^{-1} I, \epsilon_{\omega} I]^m \subset  \mathbb{P}^{m}_n$, the equation $\Phi_g(X) =0$ has a unique solution in $[\epsilon_{\omega}^{-1} I, \epsilon_{\omega} I] \subset  \mathbb{P}_n$.
\end{theorem}

A similar proof can be given for Lemma~\ref{ineq_1} and Theorem~\ref{T:alt}, analogous to the arguments used in Lemma~\ref{ineq} and Theorem~\ref{T:-spe-epsilon}. Therefore, we skip the proof.

\textbf{Open problem:}
For any $\omega \in \Delta_m$ and $\mathbb{A} = (A_{1}, \ldots, A_{m}) \in \mathbb{P}_{n}^{m}$, we know that the equation \eqref{phi_g} admits at least one solution. In the particular case $g(x)=x-1$, the solution is unique and coincides with the Wasserstein mean. In contrast, for $g(x)=\log x$, uniqueness generally fails.

This leads to the following question: determine necessary and sufficient conditions on $g \in \mathcal{L}$ under which the equation \eqref{phi_g} admits a unique solution.

%This is equivalent to
%\begin{equation}\label{E:Wass}
%X = \sum_{i=1}^{m} w_{i} (X^{1/2}A_{i}X^{1/2})^{1/2} .
%\end{equation}

\vspace{1cm}

\textbf{Acknowledgement}

The work of S. Kim was supported  by the National Research Foundation of Korea (NRF) grant funded by the Korea government (MSIT) (No. NRF-2022R1A2C4001306).  This work of V. N. Mer  was supported by Basic Science Research Program through the National Research Foundation of Korea (NRF) funded by the Ministry of Education, Korea (No. RS-2024-00462498).

\bibliographystyle{amsplain} % Choose a style for your bibliography
% \bibliography{references} % Specify the name of your .bib file

\end{document}